\documentclass[UTF8]{amsart}
\usepackage{amsmath,amssymb,graphicx,bbm,amsthm,subfigure}
\usepackage[american]{babel}
\usepackage{color}
\usepackage[all]{xy}
\newtheorem{theorem}{Theorem}[section]
\newtheorem{definition}[theorem]{Definition}
\newtheorem{lemma}[theorem]{Lemma}
\newtheorem{proposition}[theorem]{Proposition}

\newtheorem{cor}[theorem]{Corollary}
\newtheorem{question}{Question}

\newtheorem{remark}{Remark}
\numberwithin{equation}{section}

\DeclareMathOperator{\sys}{sys}

\DeclareMathOperator{\arcsinh}{arcsinh}
\DeclareMathOperator{\Cat}{Cat}
\DeclareMathOperator{\Crit}{Crit}
\usepackage{xcolor}
\usepackage{verbatim}

\usepackage{listings}
\usepackage{color} 
\definecolor{mygreen}{RGB}{28,172,0} 
\definecolor{mylilas}{RGB}{170,55,241}

\begin{document}

\lstset{language=Matlab,
    breaklines=true,
    morekeywords={matlab2tikz},
    keywordstyle=\color{blue},
    morekeywords=[2]{1}, keywordstyle=[2]{\color{black}},
    identifierstyle=\color{black},
    stringstyle=\color{mylilas},
    commentstyle=\color{mygreen},
    showstringspaces=false,
    numbers=left,
    numberstyle={\tiny \color{black}},
    numbersep=9pt, 
	xleftmargin=1cm
}

\title{An upperbound for the number of critical points of the systole function on the surface moduli space }

%
\author{Yue Gao}
\address{BICMR, Peking University, Beijing 100871, CHINA}
\email{yue\_gao@pku.edu.cn}

%
\begin{abstract}
		We obtain an upper bound for the number of critical points of the systole function on $\mathcal{M}_g$. 
		Besides, we obtain an upper bound for the number of those critical points whose systole is smaller than a constant. 
\end{abstract}

\date{}

\maketitle

\tableofcontents

\section{Introduction}


The systole of a closed hyperbolic surface is the length of the shortest geodesic on the surface. It can also be treated as a function on the Teichm\"uller space $\mathcal{T}_g$, denoted as $\sys_g: \mathcal{T}_g \to \mathbb{R}_+$. The systole function is also well-defined on the moduli space $\mathcal{M}_g$. 

The systole function on $\mathcal{T}_g$ or $\mathcal{M}_g$ has been shown to be a topological Morse function (for its definition, see Section \ref{sec_pre}) by Akrout \cite{akrout2003singularites}. Similarly to a Morse function, a topological Morse function has regular and critical points, and 
critical points of a topological Morse function on a space provide information of the space's topology. For example, let $f:M\to \mathbb{R}$ be a Morse function on a differential manifold $M$ (or a topological Morse function $f: M \to \mathbb{R}$ on a topological manifold $M$ ), then the Euler characteristic $\chi(M)$ of $M$ is given by 
\[
		\chi(M) = \sum_k (-1)^k\#\left\{ \text{critical points of }f\text{ with index }k \right\}.
\]


\begin{remark}
		For the moduli space $\mathcal{M}_g$, Harer and Zagier obtained its orbifold Euler characteristic: 
		\[
				\chi(\mathcal{M}_g) = \frac{B_{2g}}{4g(g-1)} 
		\]
		in \cite{harer1986euler}. Here $B_{2g}$ is the $2g$-th Bernoulli number. 
\end{remark}

The study on critical points of systole function initiated in a series of work by Schmutz \cite{schmutz1993reimann}, \cite{schmutz1994systoles}, \cite{schaller1998geometry}, \cite{schaller1999systoles}. He proved that systole is a topological Morse function on the space of cusped surfaces before Akrout and classified critical points of systole function on genus $2$ moduli space $\mathcal{M}_2$. Besides he provided several series of critical points, including locally maximal points. 

Recently, Fortier-Bourque and Rafi obtained new examples of critical points (including locally maximal points) of systole function \cite{fortier2019hyperbolic} \cite{fortier2020local}. Part of their examples have relatively large systole and part of theirs have small indices. Critical points with such properties had not been known before. 

Examples of critical points and the Euler characteristic $\chi(\mathcal{M}_g)$ provide lower bounds for the number of critical points of systole function. The largest lower bound is given by $\chi(\mathcal{M}_g)$ and asymptotic to $\frac{\sqrt{\pi }}{\sqrt{g}(g-1)}\left(\frac{g}{\pi e}\right)^{2g}$. Some examples offer lower bound for number of critical points with some restrictions. 

The main result of this paper is an upper bound of this number. 
We let $\sys_g$ be the systole function and $\Crit(sys_g)$ be the set of its critical points. 
		 For $L>0$ we let  
		 \[
				 \Crit(\sys_g)^{\le L} = \left\{ \Sigma\in \Crit(\sys_g)\subset \mathcal{M}_g|
		 \sys(\Sigma)\le L\right\}. 
		 \]
		 Then we get an upper bound for $|\Crit(\sys_g)^{\le L}|$ when $g\ge 2$ (Proposition \ref{prop_upper_bound}). When $g$ is sufficiently large, we have

\setcounter{section}{3}
\setcounter{theorem}{4}
 \begin{theorem}
		 If $g$ is large enough then the number of points in $\Crit(\sys_g)$ on $\mathcal{M}_g$ is bounded from above by
		 \[
				 |\Crit(\sys_g)| \le (6g)^{Cg^{\frac{7}{2}}}. 
		 \]
		 Here $C= 6054$. 
		 \label{thm_upper}
 \end{theorem}

 Also we have a theorem on how the critical points distribute on $\mathcal{M}_g$. 

 \begin{theorem}
		 When $g$ is sufficiently large, then 
		 \[
				 |\Crit(\sys_g)^{\le L}| 
				 \le \left ( 2^{C_1 ge^L}\cdot 2^{C_2 ge^{\frac{5L}{4}}}\right ) \cdot \left( (e^{L})^{C_3ge^L}\cdot (e^{L})^{C_4ge^{\frac{5L}{4}}} \right) \cdot g^{C_3ge^L}. 
		 \]
		 Here $C_1 = 200000$, $C_2 = 5040$, $C_3 = 15400$, and $C_4 = 1300$. 
		 \label{thm_const}
 \end{theorem}

\setcounter{section}{1}
\setcounter{theorem}{0}


Here is a comparison of our upper bounds with some known lower bounds. Our upper bound for $|\Crit(\sys_g)^{\le L}|$ is $(Cg)^{C'g}$ as $g$ grows for some $C=C(L)>0$ and $C'=C'(L)>0$. On the other hand, there is a sereis $\left\{ g_k \right\}$, $g_k\to \infty$ as $k\to \infty$, and there is a known lower bound for $|\Crit(\sys_{g_k})^{\le L}|$ , also is $(Cg_k)^{C'g_k}$ as $k$ grows for some other $C,C'>0$ (see \cite[Theorem B]{fortier2020local}, Theorem \ref{thm_rafi} and Corollary \ref{cor_const}). 
 As for $|\Crit(\sys_g)|$, the gap between our upper bound ($(Cg)^{C'g^{\frac{7}{2}}}$) and the lower bound implied by $\chi(\mathcal{M}_g)$ is not that small. 
 Currently we still do not have much knowledge on critical points with large systole to obtain a sharper estimate to $|\Crit(\sys_g)|$. 


 In an unpublished manuscript \cite{thurston1986spine}, Thurston claimed that let $X_g\subset \mathcal{M}_g$ be the space consisting of surfaces with filling shortest geodesics, then $X_g$ is a spine (deformation retract) of  $\mathcal{M}_g$ (See also \cite{ji2014well} \cite{fortier2019hyperbolic}, \cite{fortier2020local}). Here by filling we mean that the union of the shortest geodesics cut the surfaces into polygonal disks. Unfortunately, recently his proof was shown incomplete \cite{ji2014well}. 
 Thurston's claim is still unsolved. 
  Critical points of the systole function are contained in  $X_g$ \cite[Corollary 20]{schaller1999systoles}. and our result characterizes critical points of the systole function. 

 We use two combinatorial properties in \cite{schaller1999systoles} (Corollary 20, 21) to obtain our result. These two properties tell us that by bounding number of filling graph (with restrictions to number of vertices, edges and degree at each vertex) from above, we can bound $|\Crit(\sys_g)|$ from above. 

 The method to bound number of filling graphs on surfaces from above is similar to Kahn-Markovic's method to bound number of triangulations on surfaces from above in \cite{kahn2012counting}. 

 In Section \ref{sec_pre}, some priliminaries we need in our proof are offered. Then we prove Theorem \ref{thm_upper}  and Theorem \ref{thm_const} in Section \ref{sec_upper}. 
 In Section \ref{sec_questions}, we give some further questions to consider. 

 {\bf Acknowledgement: } The author would like to acknowledge Prof. Yi Liu for many helpful discussions.

 \section{Preliminaries}
 \label{sec_pre}

 \subsection{Topological Morse function }
 Let $M^n$ be a $n$-dimensional topological manifold. 

 \begin{definition}
		 A function $f:M^n\to \mathbb{R}$ is a topological Morse function if at each point $p\in M$, there is a neighborhood $U$ of $p$ and a map $\psi: U\to \mathbb{R}^n$. Here $\psi$ is a homeomorphism between $U$ and its image, such that $f\circ \psi^{-1}$ is either a linear function or \[
				 f\circ \psi^{-1}((x_1,x_2,\dots,x_n)) = f(p) -x_1^2-x_2^2-\dots-x_j^2+x_{j+1}^2+\dots+x_n^2.
		 \]
		 In the former case, the point $p$ is called a regular point of $f$, while in the latter case the point $p$ is called a singular point with index $j$. 
 \end{definition}

 \subsection{Teichm\"uller space and length function}

We denote by $\mathcal{T}_g$  the Teichm\"uller space consisting of marked hyperbolic surface with genus $g$, and $\mathcal{M}_g$ to be the moduli space consisting of hyperbolic surface with genus $g$. It is known that 
\[
		\mathcal{M}_g \cong \mathcal{T}_g/\Gamma_g.
\] Here $\Gamma_g$ is the mapping class group. 


For $\Sigma\in \mathcal{T}_g$, 
the set of all the shortest geodesics is denoted by $S(\Sigma)$. The length of an essential curve $\alpha$ on $\Sigma$ is denoted by $l_\alpha(\Sigma)$. For $\alpha\in S(\Sigma)$, 
\[
		l_\alpha(\Sigma)\le l_\beta(\Sigma), \forall \text{ simple closed geodesic } \beta\subset \Sigma. 
\]

  Length of the shortest geodesics can be treated as a function on $\mathcal{T}_g$, we denote it as $\sys_g$ or shortly $\sys$. Obviously for $\alpha\in S(\Sigma)$
 \[
		 \sys(\Sigma) = l_\alpha(\Sigma) = \inf_{\forall\text{ simple closed geodesic }\beta \subset \Sigma }  l_\beta(\Sigma). 
 \]


 Since systole is an invariant function on $\mathcal{T}_g$ under the action of the mapping class group, systole function can be also defined as a function on $\mathcal{M}_g$. 
\begin{eqnarray*}
		\sys: \mathcal{M}_g &\to & \mathbb{R}^+\\
		\Sigma & \mapsto &\sys(\Sigma).
\end{eqnarray*}
Akrout showed that 

 \begin{cor}[\cite{akrout2003singularites}, Corollary]
		 The systole function is a topological Morse function on $\mathcal{T}_g$.  
 \end{cor}
 Therefore systole function is also a topological Morse function on $\mathcal{M}_g$.  


 The set of all the critical points of $\sys_g$ in $\mathcal{M}_g$ is denoted by $\Crit(\sys_g)$. The number of the critical points in $\mathcal{M}_g$, is denoted by $|\Crit(\sys_g)|$. 

 \subsection{Filling curves}
 \begin{definition}
		 For a closed surface $\Sigma$, we assume $F=\{ \alpha_1,\alpha_2,\dots,\alpha_n\}$ is a finite set of essential simple closed curves. $F$ is said to be filling if and only if each component of $\Sigma \backslash F$ is a disk. 
 \end{definition}

 We need criteria regarding critical points of the systole function provided by the following two corollaries in \cite{schaller1999systoles}. 

 \begin{cor}[\cite{schaller1999systoles} Corollary 20]
		 If $\Sigma \in \mathcal{T}_g$ is a critical point of the systole function, then $S(\Sigma)$ is a filling set.  
		 \label{cor_schmutz_fill}
 \end{cor}

 \begin{definition}
		 If $F$ and $F'$ are filling sets on $\Sigma$ and $\Sigma'$ respectively, then $(\Sigma,F)$ is  {\emph equivalent} to $(\Sigma',F')$  if and only if there is a homeomorphism from $\Sigma$ to $\Sigma'$ that maps the elements of $F$ to those of $F'$. 
 \end{definition}

 \begin{cor}[\cite{schaller1999systoles} Corollary 21]
		 For $\Sigma, \Sigma'\in \mathcal{T}_g$, if $(\Sigma,S(\Sigma))$, $(\Sigma',S(\Sigma'))$ are equivalent then $\Sigma$ and $\Sigma'$ cannot both be critical points of the systole function. 
		 \label{cor_schmutz_unique}
 \end{cor}

 \section{Upper bound}
\label{sec_upper}

 On a $\Sigma\in \Crit(\sys)$, the union of shortest geodesics can be treated as a graph. Vertices of this graph are the intersection points of shortest geodesics. Edges of the graph consist of subarcs of shortest geodesics connecting two intersection points of shortest geodesics. This graph is denoted by $\Gamma$. 

 To obtain the upper bound of $|\Crit(\sys_g)|$ on $\mathcal{M}_g$, first we need to estimate the number of vertices, edges and degree at each vertex of the graph $\Gamma$ defined in the previous paragraph. 
 Then we count the number of graphs satisfying this restriction by the method of Kahn-Markovic in \cite{kahn2012counting}. 
 This number bounds  $|\Crit(\sys_g)|$ from above because of Corollary \ref{cor_schmutz_fill} and \ref{cor_schmutz_unique}. 

 It is known that two shortest geodesics intersect at most  once, see for example \cite{parlier2014simple}. 

 By this fact, we have the folloing estimate:

 \begin{lemma}
		 For any $\Sigma\in \Crit(\sys_g)$, the graph $\Gamma=(V,E)$ satisfies
		 \begin{enumerate}
				 \item $|V|\le 2545ge^{\sys(\Sigma)}$, 
				 \item $|E|\le  3|V|+6g-6 \le 7700 ge^{\sys(\Sigma)}$, 
				 \item The genus 
						 \footnote{For a connected graph $\Gamma = (V,E)$, the genus of $\Gamma$ is the number of circles of $\Gamma$, more precisely, is $\min_{E_0\in \mathcal{E}} |E_0|$. Here $\mathcal{E} = \left\{ E_0 \subset E |\Gamma\backslash E_0 \text{ is a tree. } \right\}$.}
						 of $\Gamma$ (denoted by $g_{\Gamma}$) is bounded by $\pi\frac{\sqrt{g(g-1)}}{\log(4g-2)}\le g_{Gr}\le 3|V|+6g-6\le 7700 ge^{\sys(\Sigma)}$. 
				 \item For all $ v\in V$, $\deg(v)\le 2e^{\sys(\Sigma)/4}$. 
		 \end{enumerate}
		 \label{lem_gph_control}
 \end{lemma}

 \begin{proof}
		 (1) In \cite{parlier2013kissing}, for a surface $\Sigma\in \mathcal{M}_g$, Parlier defined $F(\Sigma)$ to be the minimum number of radius-$r$-disks  needed for covering $\Sigma$ and $G(\Sigma)$ to be supremum number of shortest geodesics of $\Sigma$ that pass through a radius-$r$-disk, where $r:=\arcsinh\left( \frac{1}{2\sinh(\sys(\Sigma)/4)} \right)$. 
		 These two numbers are bounded from above by
		 \[
				 F(\Sigma) < 16(g-1)e^{\sys(\Sigma)/2}  \text{ and } G(\Sigma)< 17.83e^{\sys(\Sigma)/4}.
		 \]
 In each radius-$r$-disk, there are at most ${G(\Sigma)}\choose{2}$ intersection points since it is known that two shortest geodesics intersect at most once. Then on the surface $\Sigma$, number of intersection points of shortest geodesics is bounded from above by 
		 \[
				 \frac{F(\Sigma)G^2(\Sigma)}{2}. 
		 \]
		 Then (1) follows by direct calculation. 

		 (2) We can add some edges to $\Gamma$ to get a triangulation of $\Sigma$. Therefore we can bound the number of edges of $\Gamma$ from above by bounding the number of edges of the triagulations containing $\Gamma$ from above. 

		 If $\Gamma$ is a triagulation of $\Sigma$, then the number of edges and faces of this triangulation satisfies the following relation: 
		 \begin{equation}
				 2E=3F.
				 \label{for_edge_face}
		 \end{equation}
The Euler characteristic formula is 
\begin{equation}
		V-E+F =2-2g. 
		\label{for_euler}
\end{equation}

Then by (\ref{for_edge_face}) and (\ref{for_euler}), (2) follows. 

		 (3) The genus of a graph is $g_{\Gamma} = 1- \chi(\Gamma)$. Here $\chi(\Gamma)$ is the Euler characteristic of $\Gamma$. Then the upper bound follows from (1) and (2). 

		 By \cite[Theorem 3]{anderson2011small}, if $\Sigma\in \mathcal{M}_g$ has filling shortest geodesics then $|S(\Sigma)|\ge \pi\frac{\sqrt{g(g-1)}}{\log(4g-2)}$. Therefore $\pi\frac{\sqrt{g(g-1)}}{\log(4g-2)}$ also lower bounds $g_{\Gamma}$. 

		 (4) This follows directly from \cite[Lemma 2.4]{parlier2013kissing}, for details, see \cite{parlier2013kissing}. 
 \end{proof}

 We may add some vertices to $\Gamma$ such that (1) every edge of the resulting graph connects two different vertices, (2) number of vertices, edges, and degree of the resulting graph is still bounded by the bounds in Lemma \ref{lem_gph_control}. We still denote the resulting graph by $\Gamma$ . 

 By Corollary \ref{cor_schmutz_fill}, each component of $\Sigma\backslash \Gamma$ is a polygonal disk. 
 We add some edges to $\Gamma$ so that we get a triangulation of $\Sigma$ and the number of vertices, edges and degree at each vertex of the triangulation are bounded by the bounds in Lemma \ref{lem_gph_control}. (See Figure \ref{fig_polygon_1}. )

 \begin{figure}[htbp]
		 \centering
		 \includegraphics{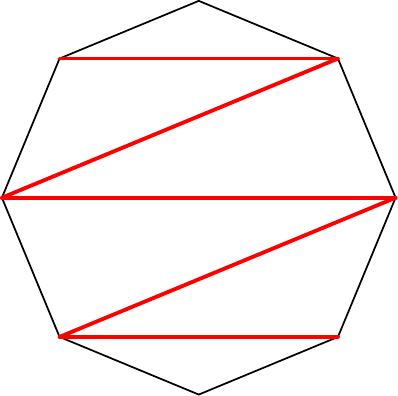}
		 \caption{A triangulation of the polygon}
		 \label{fig_polygon_1}
 \end{figure}

 We recall a definition in \cite{kahn2012counting}:
 \begin{definition}[\cite{kahn2012counting}, Definition 2.1]
		 For $k>0$, $g\ge2$, a triangulation $\tau$ of genus $g$ surface is in the set $T(k,g)$ if and only if 
		 \begin{itemize}
				 \item each vertex of $\tau$ has the degree at most $k$, 
				 \item the graph $\tau$ has at most $kg$ vertices and at most $kg$ edges.
		 \end{itemize}
 \end{definition}

 By the construction above, we have
 \begin{lemma}
		 There is a constant $0<C<7700$, for all $ \Sigma\in \Crit(\sys_g)$, $\Gamma$, the graph consisting of shortest geodesics of $\Sigma$ is a subgraph of a triangulation $\tau \in T(Ce^{\sys(\Sigma)},g)$. 
		 \label{lem_subgraph}
 \end{lemma}

 By Corollary \ref{cor_schmutz_unique}, shortest geodesics on two different critical points of $\sys_g$ must be non-equivalent. 
 Therefore number of critical points of $\sys_g$ whose systole length is smaller than $L>0$ is bounded from above by number of filling subgraphs of triangulations in $T(Ce^L,g)$ that satisfies the restriction of Lemma \ref{lem_gph_control}. 

 \begin{proposition}

		 For $L>0$ we let  
		 \[
				 \Crit(\sys_g)^{\le L} = \left\{ \Sigma\in \Crit(\sys_g)\subset \mathcal{M}_g|
		 \sys(\Sigma)\le L\right\}, 
		 \]
		 and let $|V_0| = 2545ge^L$, 
		 $|g_{\Gamma_0}| = 7700ge^L$, and $\deg_0 = 2 e^{L/4}$. 

Then 
\begin{equation}
		|\Crit(\sys_g)^{\le L}| \le \Cat(|V_0|-1) \frac{(|V_0|)^{2g_{\Gamma_0}}}{\left \lfloor \pi\frac{\sqrt{g(g-1)}}{\log(4g-2)}\right \rfloor  !} (\lceil\deg_0\rceil!)^{|V_0|}. 
		\label{for_crit}
\end{equation}
Here $\Cat(|V|-1)$ is the $(|V|-1)$-th Catalan number. 
\label{prop_upper_bound}
 \end{proposition}

 \begin{proof}
		The proof here is similar to the proof of \cite[Lemma 2.2]{kahn2012counting}. 

		Fix some genus $g$ surface, we assume that 
		\[
				\mathcal{G} = \left\{ \Gamma|\exists \tau \in T(Ce^{L}, g), s.t. \Gamma\subset \tau, \Gamma \text{ is filling, }\Gamma \text{ satisfies the restriction in Lemma \ref{lem_gph_control}} \right\}.
		\]

		Any $\Gamma \in \mathcal{G}$ can be constructed in this way. 
		We pick a $\tau \in T(Ce^L,g)$, and then choose a spanning tree $T$ of $\tau$. Since $\tau \in T(Ce^L,g)$, $T$ has at most $2545ge^L$ vertices. After that, we add $g_{\Gamma}$ edges $e_1,e_2,\dots,e_{g_{\Gamma}}$ to the tree $T$ in an arbitary way. Thus $T\cup \left\{ e_1,\dots,e_{g_{\Gamma}} \right\}$ is a graph with genus at most $g_{\Gamma}$. Next we give each vertex of $T\cup \left\{ e_1,\dots,e_{g_{\Gamma}} \right\}$ a cyclic ordering, thus $T\cup \left\{ e_1,\dots,e_{g_{\Gamma}} \right\}$ is a {\em ribbon graph}\footnote{A ribbon graph is a graph with every vertex of the graph assigned a cyclic ordering. }. We thicken the edges of the ribbon graph and get a surface with boundary. If the genus of this surface is not $g$, we discard it. Finally, we fill each boundary component of the surface by a disk, thus get a closed surface $S_g$. 
		Every graph in $\mathcal{G}$ can be constructed this way. 

		By this construction, $|\mathcal{G}|$ can be bounded from above by $|\mathcal{G}|\le abc$. Here 
		\begin{eqnarray*}
				a&=&\left\{ \text{number of unlabelled trees }T \text{ with }n\le 2545ge^L \text{ vertices} \right\}\\
				b&=& \left\{ \text{number of ways of adding } g_{\Gamma} \text{ many unlabelled edges } e_1,\dots,e_{g_{\Gamma}} \text{ to } T \right\} \\
				c&=& \left\{ \text{number of cyclic orderings of edges of } T\cup \left\{ e_1,\dots,e_{g_{\Gamma}} \right\}\right\}. 
		\end{eqnarray*}
		Because Catalan number $\Cat(n)$ enumerates number of ordered planar trees with $n+1$ vertices (see for example \cite{stanley2015catalan}) and by proof of \cite[Lemma 2.2]{kahn2012counting} and Lemma \ref{lem_gph_control}, 
		\begin{eqnarray*}
				a &\le& \Cat(|V_0|-1) \\
				b &\le& \frac{(|V_0|)^{2g_{\Gamma_0}}}{\left \lfloor \pi\frac{\sqrt{g(g-1)}}{\log(4g-2)}\right \rfloor  !}
				 \\
				c &\le& (\lceil\deg_0\rceil!)^{|V_0|}. 
		\end{eqnarray*}
		Then this proposition follows. 
 \end{proof}

 By area comparison, for any $\Sigma\in \mathcal{M}_g$, $\sys(\Sigma)\le \log(2g^2)$. 
 Therefore if we let $L = \log(2g^2)$, then $\Crit(\sys_g) = \Crit(\sys_g)^{\le L}$. We let $g\to \infty$, then we have the followiing theorem: 
 \begin{theorem}
		 If $g$ is large enough then 

		 \[
				 |\Crit(\sys_g)| \le (6g)^{Cg^{\frac{7}{2}}}. 
		 \]
		 Here $C= 6054$. 
		 \label{thm_upper}
 \end{theorem}

 \begin{proof}
		 This theorem follows from direct calculation. 
		 First we let $L= \log(2g^2)$ and then we obtain formulae of $|V_0|$, $g_{\Gamma_0}$ and $\deg_0$ in Proposition \ref{prop_upper_bound} in $g$: 
		 \begin{eqnarray*}
				 |V_0|&=& 2545ge^L = 5090g^3\\
				 g_{\Gamma_0} &=& 7700 ge^L = 15400g^3\\
				 \deg_0 &=& 2e^{\frac{L}{4}} = 2^{\frac{5}{4}} g^{\frac{1}{2}}.
		 \end{eqnarray*}

		 When $n$ is sufficiently large, $\Cat(n) \sim \frac{4^n}{\sqrt{\pi}n^{\frac{3}{2}}}$, see for example \cite{stanley2015catalan}. Thus when $g$ is large enough, 
		 \begin{equation}
				 \Cat(|V_0|-1) \le \frac{4^{|V_0|-1}}{\sqrt{\pi} (|V_0|-1)^{\frac{3}{2}}} 
				 \le 4^{|V_0|-1} 
				 \le 4^{5090g^3}. 
				 \label{for_catalan}
		 \end{equation}
		 Besides 
		 \begin{equation}
				 \frac{(|V_0|)^{2g_{\Gamma_0}}}{\left \lfloor \pi\frac{\sqrt{g(g-1)}}{\log(4g-2)}\right \rfloor  !} \le
				 (|V_0|)^{2g_{\Gamma_0}} = \left( 5090g^3 \right)^{15400g^3}, 
				 \label{for_b_item}
		 \end{equation}
		 and
		 \begin{eqnarray}
				 (\lceil\deg_0\rceil!)^{|V_0|} &\le & \deg_0^{\deg_0 |V_0|} 
				 \label{for_c_item}\\
				 &=& \left( 2^{\frac{5}{4}}g^{\frac{1}{2}} \right)^{2^{\frac{5}{4}}g^{\frac{1}{2}} \cdot 5090g^3} \nonumber\\
				 &\le& (5.6g)^{6053.1g^{\frac{7}{2}}}.\nonumber `
		 \end{eqnarray}
		 Then by (\ref{for_crit}), (\ref{for_catalan}), (\ref{for_b_item}) and (\ref{for_c_item}), when $g$ is large enough, 
		 \[
				 |\Crit(\sys_g)| \le (6g)^{6054g^{\frac{7}{2}}}. 
		 \]
 \end{proof}

 If $L>0$ is a constant, we have 

 \begin{theorem}
		 If $L>0$ is a constant, $g$ is sufficiently large, then 
		 \[
				 |\Crit(\sys_g)^{\le L}| 
				 \le \left ( 2^{C_1 ge^L}\cdot 2^{C_2 ge^{\frac{5L}{4}}}\right ) \cdot \left( (e^{L})^{C_3ge^L}\cdot (e^{L})^{C_4ge^{\frac{5L}{4}}} \right) \cdot g^{C_3ge^L}. 
		 \]
		 Here $C_1 = 200000$, $C_2 = 5040$, $C_3 = 15400$, and $C_4 = 1300$. 
		 \label{thm_const}
 \end{theorem}
 \begin{proof}
		 If $L$ is a constant, $g$ is sufficiently large, then similarly to Theoorem \ref{thm_upper}, we have
		 \begin{eqnarray}
				 \label{for_catalan_l}\Cat(|V_0|-1)&\le& 4^{|V_0|-1} \\
				 &=& 4^{2545ge^L-1} \nonumber\\
				 \label{for_b_item_l}\frac{(|V_0|)^{2g_{\Gamma_0}}}{\left \lfloor \pi\frac{\sqrt{g(g-1)}}{\log(4g-2)}\right \rfloor  !} &\le& (|V_0|)^{2g_{\Gamma_0}} \le \left( 2545ge^L \right)^{7700ge^L\cdot 2} \\
				 &=&\left( 2545ge^L \right)^{15400ge^L}\nonumber  \\
				 \label{for_c_item_l}(\lceil\deg_0\rceil!)^{|V_0|} &\le & \deg_0^{\deg_0 |V_0|} \le \left( 2e^{\frac{L}{4}} \right)^{2e^{\frac{L}{4}}\cdot 2545ge^L} \\
				 &=& \left( 2e^{\frac{L}{4}} \right)^{5040ge^{\frac{5L}{4}}}. \nonumber 
		 \end{eqnarray}

		 Then by (\ref{for_crit}), (\ref{for_catalan_l}), (\ref{for_b_item_l}) and (\ref{for_c_item_l}), when $g$ is large enough, 
		 \[
				 |\Crit(\sys_g)^{\le L}| \le \left( 2546ge^L \right)^{15401ge^{L}}. 
		 \]
		 \begin{eqnarray*}
				 |\Crit(\sys_g)^{\le L}| &\le& 4^{2545ge^L-1}\cdot \left( 2545ge^L \right)^{15400ge^L} \cdot \left( 2e^{\frac{L}{4}} \right)^{5040ge^{\frac{5L}{4}}} \\
				 &\le&\left( 2^{5090ge^L}\cdot (2^{12})^{15400ge^L} \cdot 2^{5040ge^{\frac{5L}{4}}}\right)\cdot \\
				 && \left( (e^{L})^{15400ge^L}\cdot (e^{\frac{L}{4}})^{5040ge^{\frac{5L}{4}}} \right) \cdot g^{15400ge^L} \\
				 &\le& \left ( 2^{189890ge^L}\cdot 2^{5040 ge^{\frac{5L}{4}}}\right ) \cdot \left( (e^{L})^{15400ge^L}\cdot (e^{L})^{1273ge^{\frac{5L}{4}}} \right) \cdot g^{15400ge^L}. 
		 \end{eqnarray*}
 \end{proof}

 Fortier-Bouque and Rafi \cite{fortier2020local} obtained a series of local maxima of the systole function. Letting $n=3$ in \cite[Theorem B]{fortier2020local}, we have 
 \begin{theorem}[\cite{fortier2020local}, Theorem B]
		 There is $\beta>0$ and a sequence $g_k$, $g_k\to \infty$ as $k\to \infty$ such that 
		 \[
				 \left | \Crit(\sys_{g_k})^{\le 10} \right |\ge \left( \beta g_k \right)^{\frac{1}{3}g_k}.
		 \]
		 \label{thm_rafi}
 \end{theorem}

 This lower bound is not far from the upperbound we obtain.
 Letting $L=10$ in Theorem \ref{thm_const}, we immediately get the following corollary: 
 \begin{cor}
		 There are constants $\beta,\beta'>0$ and a sequence $g_k$, $g_k\to \infty$ as $k\to \infty$ such that
		 \[
				 \left( \beta g_k \right)^{\frac{1}{3}g_k}\le\left | \Crit(\sys_{g_k})^{\le 10} \right |\le \left( \beta' g_k \right)^{4\times 10^8 g_k}.
		 \]
		 \label{cor_const}
 \end{cor}

 \section{Further questions}
 \label{sec_questions}
 We have some further questions to consider: 

		 With the restriction $\sys\le C$, the upper bound is quite close to the lower bound. Then
 \begin{question}
		 Could we obtain a sharper estimate to critical points with large systoles?
 \end{question}
 Moreover, 
 \begin{question}
		 For all $C>0$, does
		 \[
				 \frac{\left | \Crit^{<C}(\sys_g)\right |}{\left | \Crit(\sys_g)\right |}
		 \] tend to $0$ as $g\to \infty$?
 \end{question}
 The lower bound in \cite[Theorem B]{fortier2020local} is provided merely by locally maximal points. Therefore, 
 \begin{question}
		 Could we estimate the number of critical points of the systole function with a specific index?
 \end{question}
 On the distribution of the systole function: 
 \begin{question}
		 Let $l_g:= \inf_{\Sigma\in \Crit(\sys_g)} \sys_g(\Sigma)$, by \cite[Corollary 20]{schaller1999systoles} and Mumford compactness criterion, $l_g$ has a positive universal lower bound. Can we estimate $l_g$? 
 \end{question}

\bibliographystyle{alpha}
\bibliography{systole_euc}   

\end{document}